\documentclass{article}
    \usepackage[utf8]{inputenc}
    \usepackage{verbatim,color,amsmath}
    \usepackage{amssymb}
    \usepackage{amsthm}
    \usepackage{mathrsfs,leftidx}
    \usepackage[shortlabels]{enumitem}
    
    \newtheorem{thm}{Theorem}[section]
    \newtheorem{prop}[thm]{Proposition}
    \newtheorem{lemma}[thm]{Lemma}
    \newtheorem{cor}[thm]{Corollary}
    
    \newtheorem{observation}[thm]{Observation}
    \newtheorem{ex}[thm]{Example}
    \newtheorem{remark}[thm]{Remark}

    \newcommand{\stb}{, \ldots ,}
    
    \newcommand{\N}{{\mathbb{N}}}
    
    \newcommand{\Q}{{\mathbb{Q}}}
    \newcommand{\R}{{\mathbb{R}}}
    \newcommand{\C}{{\mathbb{C}}}

    \def\bl{\textcolor{blue}}

    \makeatletter
    \newcommand{\subjclass}[2][1991]{%
      \let\@oldtitle\@title%
      \gdef\@title{\@oldtitle\footnotetext{#1 \emph{MSC classes.} #2}}%
    }
    \newcommand{\keywords}[1]{%
      \let\@@oldtitle\@title%
      \gdef\@title{\@@oldtitle\footnotetext{\emph{Keywords:} #1.}}%
    }
    \makeatother

\begin{document}

    \title{Decompositions of the positive real numbers into disjoint sets closed under addition and multiplication}
    \author{Gergely Kiss \thanks{Alfr\'ed R\'enyi Institute of Mathematics, e-mail: kigergo57@gmail.com}, G\'abor Somlai \thanks{On an unpaid leave at E\"otv\"os Lor\'and University, Department of Algebra and Number Theory and Alfr\'ed Rényi Institute of Mathematics, A Fulbright scholar at the Graduate Center of the City University of New York, e-mail: gabor.somlai@ttk.elte.hu}, Tam\'as Terpai\thanks{E\"otv\"os Lor\'and University, Department of Analysis}}
    \keywords{derivation, decomposition, closed under addition and multiplication}
    
    \subjclass[2020]{11C08, 11P99, 13N15, 13P05}

    \maketitle
    \begin{abstract}
    The main purpose of this paper is to prove that the positive real numbers can be decomposed into finitely many disjoint pieces which are also closed under addition and multiplication. 
    As a byproduct of the argument we determine all the possible decompositions of the transcendental extension of the rational field of rank one into two pieces.
    Further, we prove that the positive elements of a real algebraic extensions of the rational numbers are indecomposable into two pieces. 
    \end{abstract}
    
    \section{Introduction}
    R\'obert Freud asked whether the positive real numbers can be cut into two disjoint sets both closed under addition and multiplication. 
    
    The analogous question can be asked for any subfield of $\R$. We denote by $K^{+}$ the positive elements of a subset $K \subset \R$.   
    We call a subset of $\R$ {\it invariant}, if it is closed under addition and multiplication.  By a {\it decomposition of} $K^{+}$ (or $K$) into $k$ subsets we mean that 
    $K^{+}$ is the union of $k$ pairwise disjoint non-empty invariant subsets. We denote this by $$K^{+}= H_1 \sqcup \ldots \sqcup H_k.$$  
    We say that a decomposition of $K^{+}$ into non-empty subsets is {\it non-trivial} if $k \ge 2$.
    Freud asked whether there exists a non-trivial decomposition of $\R^{+}$ into two pieces. 
    Finally, we say that $K^+$ is {\it $k$-decomposable} ($k \ge 2$) if there exists a non-trivial decomposition of $K^{+}$ into $k$ invariant pieces. During the preparation of this paper the authors noticed that the same problem was also discussed by Kane \cite{kane} for the positive real numbers. Note however, that our method is fundamentally different from that of Kane.
    
    Elekes and Keleti \cite{elek} investigated the decomposition of the real line $\R$ into Borel subsets which are closed under addition. 
    For a decomposition $\R = \cup_{\alpha <\kappa } A_{\alpha}$, where each $A_{\alpha}$ is a Borel set closed under addition, they proved that every $A_{\alpha}$ is of Lebesgue measure zero and is of first category, or is equal to $\R$, $(-\infty, 0)$, $(-\infty, 0]$, $(0, \infty)$, or $[0, \infty)$.
    Furthermore, if $\R$ is decomposed into the disjoint union of $\kappa$ Borel sets, all closed under addition, then $\kappa \in \{1,2,3\}$ or $\kappa$ is uncountable. 
    They also proved that it is consistent with ZFC that $\R$ has a decomposition into $\omega_1$ pieces which are all $F_{\sigma}$ and closed under addition. 
    
    
    We give simple constructions of non-trivial decompositions of the set of positive elements of transcendental extensions of $\Q$. 
    It will be shown in Section \ref{SeConstruction} that if $\alpha$ is a transcendental number, then any decomposition of $\Q(\alpha)^+$ is simply determined by a unique non-trivial derivation. Our method to answer the original question relies on the fact that such a derivation can be extended to a derivation on $\R$, see \cite{K}. 
    Since our construction is based on the decomposition of a finitely generated subfield of $\R$, it is natural to investigate all the possible decompositions of $\Q[\alpha]^+$. As a corollary presented in Section \ref{SeConstruction} we prove the following. 
    \begin{thm}\label{thmRszetvag}
    The positive real numbers can be written as the union of $k$ disjoint nonempty invariant sets for every natural number $k$.
    \end{thm}
    One can simply extend this construction to a decomposition into countable infinitely many pieces.
    
    Our main goal is to characterise the possible decompositions of $K^+$ for a subfield $K$ of $\R$. 
    We introduce the notion of separating hyperplanes, which play an important role in the decomposition of transcendental extensions of $\Q$ of degree $1$. 
    One of our main results is Theorem \ref{thm1valt}, which describes the possible 2-decompositions of $\Q(\alpha)^+$, where $\alpha$ is a transcendental number. Since the statement is somewhat technical we will only formulate the precise description in Section \ref{SeConstruction}.
    On the other hand, non-trivial finite decompositions for the algebraic extension of $\Q$ do not exist. 
    \begin{thm}\label{thmalgebraic}
    Let $a$ be an algebraic element over $\Q$. Then $\Q(a)^{+}$ is not 2-decomposable. 
    \end{thm}
    
    The paper is organised as follows. In Section \ref{SeConstruction} we construct decompositions of finite degree transcendental extensions of $\Q$. Section \ref{sec3} contains a useful general lemma (Lemma \ref{lemraceltol}) that helps describing the decompositions of $\R^+$. Further it is proved in this section that the elements of a finite decomposition of $\R^+$ cannot all be Lebesgue measurable. We describe the 2-decompositions of a polynomial ring over $\Q$ in Section \ref{secdegree1}. Such a decomposition is lifted to its field of fractions $\Q(x)$ in Section \ref{sec5}. In Section \ref{sec6} we prove that a finite extension of $\Q$ cannot be non-trivially decomposed into 
    two pieces. 
    Finally in Section \ref{sec7} we formulate some questions that are closely related to our results.

    \section{Fundamental construction}\label{SeConstruction}
    Let $K=\mathbb{Q}(\alpha)$ be a transcendental extension of $\mathbb{Q}$. Now we present the basic construction that gives a decomposition of $K^+$.
    
     We may identify the elements of $K$ with rational functions over $\mathbb{Q}$ so $K$ can be written as $\frac{p}{q}(\alpha)$, where $p$ and $q$ are in $\mathbb{Q}\left[ x\right]$. Let $\Q(\alpha)^+$ consist only of the elements $r \in \mathbb{Q}(x)$ with $r(\alpha)>0$. 
    
    Let   
    \begin{equation*} 
    \begin{split}
    &H_+ =\left\{ r \in \Q(\alpha)^+ \mid r'( \alpha ) > 0 \right\},\\
    &H_0 =\left\{ r \in  \Q(\alpha)^+ \mid r'( \alpha ) = 0 \right\},\\
    &H_- =\left\{ r \in  \Q(\alpha)^+ \mid r'( \alpha ) < 0 \right\}, 
    \end{split}
    \end{equation*}
    where $r'$ denotes the derivative of $r$. Note that $H_0=\Q^+$. 
    \begin{prop}\label{lemSeCo1}
    $H_+ \sqcup H_0 \sqcup H_-=K^{+}$ gives a 3-decomposition of $K^{+}$.
    \end{prop}
    \proof
    It is clear that these sets are closed under addition since every derivation is a linear operator. Moreover, the Leibniz rule $$(r_1r_2)'(\alpha)=r_1'(\alpha)r_2(\alpha)+r_2'(\alpha)r_1(\alpha)$$ and the fact that $r_1(\alpha), r_2(\alpha)>0$ shows that these sets are also closed under multiplication. 
    \qed
    
    \medskip
    The same argument shows that $(H_+\cup H_0) \sqcup H_-$ and $H_+ \sqcup (H_0 \cup H_-)$ are 2-decompositions of $K^+$. 
    In Section \ref{sec5} we will show that these are the only ones.
    \begin{thm}\label{thm1valt} Let $\alpha$ be a transcendental element over $\mathbb{Q}$.
    All the possible non-trivial decompositions of $K^+=\Q(\alpha)^+$ into two pieces are the following:
    \begin{equation*}
    \begin{split}
    \Q(\alpha)^+ &=(H_+ \cup H_0 )\sqcup H_-,\\
    \Q(\alpha)^+ &=H_+ \sqcup (H_0 \cup H_-).
    \end{split}
    \end{equation*}
    \end{thm}
    
    Now we generalise the previous construction to finite transcendental extensions. Let $K=\mathbb{Q}(\alpha_1, \ldots , \alpha_m)$, where $\alpha_i\in  \R ~ (i=1, \dots, m)$ are algebraically independent over $\Q$ and $m\ge 2$. The $m$-tuple $(\alpha_1, \ldots , \alpha_m )$ will also be denoted by $\underline{\alpha}$. Now $K$ can be identified with $\mathbb{Q}(x_1, \ldots, x_m)$, and $K^+$ denotes the elements $r\in \Q(x_1, \ldots, x_m)$ such that $r(\underline{\alpha})>0$. Let $dx_i$ denote the partial derivative according to $x_i$ on $K$, i.e.: 
    $$ d x_i (r)=\frac{\partial r(x_1, \ldots, x_m)}{\partial x_i} ~~ \forall r \in K,$$ and let $D=l(dx_1,dx_2, \ldots , dx_m)$ be a linear combination of partial derivatives with real 
    coefficients.
    Clearly, $D$ acts on $\mathbb{Q}(x_1,x_2, \ldots, x_m)$ as a derivation\footnote{A derivation $d:K\to \C$ is additive (i.e. $d(x+y)=d(x)+d(y)$) and satisfies the Leibniz rule, i.e. $d(xy)=xd(y)+yd(x)$}. Using an argument similar to the proof of Proposition \ref{lemSeCo1} one obtains decompositions into two or three pieces using the following sets which are invariant under addition and multiplication: 
    \begin{equation}\label{eqparts}
    \begin{split}
    H_{D,+,l} &=\left\{ r(\underline{\alpha}) \in K^{+} \mid D(r)( \underline{\alpha} ) > 0 \right\} ,\\
    H_{D,0,l} &=\left\{ r(\underline{\alpha}) \in K^{+} \mid D(r)( \underline{\alpha} )  = 0 \right\} ,\\
    H_{D,-,l} &=\left\{ r(\underline{\alpha}) \in K^{+} \mid D(r)( \underline{\alpha} )  < 0 \right\} .
    \end{split}
    \end{equation}
    \begin{prop}\label{propnresz}
    Let $\alpha_1 \stb \alpha_m$ be algebraically independent over $\Q$, where $m\ge 2$.  
    For every $2 \le n \in \N$ there exist infinitely many different ways to nontrivially decompose $K \cong \Q(\alpha_1,\dots,\alpha_m)$ into $n$ pieces. 
    \end{prop}
    \proof 
    Every derivation $D$ is a linear combination $l$ of partial derivatives $dx_i$. The derivation $D$ implies a 2-decomposition of $K^{+}$ as $K^{+}= \left( H_{D,+,l} \cup H_{D,0,l} \right) \sqcup H_{D,-,l}$. Our strategy is to take finitely many derivations such that no two of them are multiples of each other and take one of the two pieces as above for each of them. Then the intersection of these pieces with careful choice of the sign of the derivation will be a nonempty invariant set. 
    
    Now we can construct infinitely many different decompositions into $n$ pieces by taking infinitely many different sets of $n-1$ pairwise linearly independent derivations $l_1$, $\dots$, $l_{n-1}$. As an initial step  we consider the decomposition $K^+=\left( H_{D,+,l_1} \cup H_{D,0,l_1} \right) \sqcup H_{D,-,l_1}$. Then we iteratively cut one of the pieces of the previous step into two pieces. More precisely, in the $j^{th}$ step for $j\in \{1,\dots,n-1\}$ we created the pieces $H_1, \dots H_{j+1}$. Now we take $\left( H_{D,+,l_j} \cup H_{D,0,l_j} \right)$ and $ H_{D,-,l_j}$ that are  invariant sets. Clearly, there is a set $H_k$ for some $1<k<j+1$ such that $H_k \cap (H_{D,+,l_j} \cup H_{D,0,l_j})$ and $H_k\cap H_{D,-,l_j}$ are nonempty and invariant sets, as they are the intersection of invariant sets. The existence of such $H_k$ is guaranteed by the fact that $l_j$ for $j\in \{1,\dots,n-1\}$ are pairwise linearly independent and $\cup_{i=1}^{j+1} H_i=K^+$. Then instead of $H_k$ we add $H_k \cap (H_{D,+,l_j} \cup H_{D,0,l_j})$ and $H_k\cap H_{D,-,l_j}$
    to the system of invariant sets $H_i$ ($i\ne k$) increasing the number of pieces by one. \qed
    
    \medskip
    Our first goal is to find a solution so that we can extend the previously obtained decompositions of $\mathbb{Q}(\underline{\alpha})^+$ to $\mathbb{R}^+$.  
     Now we deal with the problem of extension of a derivation obtained from Proposition \ref{lemSeCo1}.
    Indeed, let us assume that there exists a nontrivial decomposition of the positive elements of a finitely generated field $K\subset \R$ over $\Q$, where the pieces are determined by inequalities involving derivations in the manner of \eqref{eqparts} and Proposition \ref{propnresz}. 
    
    The idea is that extending (not necessarily uniquely) these derivations to $\R$ automatically gives a partition of $\R^+$ and the number of pieces stays the same as for the decomposition of $K^+$.
    Indeed, let $K$ be an arbitrary field over $\Q$. Let $a$ be an algebraic element over $K$. Then by \cite[Lemma 14.2.2]{K}, a derivation on $K$ extends uniquely to $K(a)$. Let $S\subset K$ be a set of algebraically independent elements over $K$. \cite[Lemma 14.2.3]{K} shows that if $d : K \to \R$ is a derivation, and $u :S \to \R$ is an arbitrary function, then there exists a unique derivation $\hat{d} : F(S) \to \R$
    such that $\hat{d} |_K = d$ and $\hat{d} |_S = u$.
    
    Let $D_i$ be derivations on $K= \mathbb{Q}(t_1, \dots, t_m)~ (m\ge 2)$ determined by $l_i$ which give a decomposition into $n$ parts. The previously discussed extensions imply that every derivation $D_i$ from a subfield $K$ of $\R$ can be extended to $\R$. A standard transfinite recursion verifies the existence of the extensions, which proves Theorem \ref{thmRszetvag}. 
    
    Let $H_{D_i,+,l_i}, H_{D_i,0,l_i}, H_{D_i,-,l_i}$ be defined by \eqref{eqparts} where $K=\mathbb{R}$. The existence of these extended derivations immediately imply the existence of a nontrivial decomposition of $\R^+$ into $n$ pieces. Note that no new pieces were constructed. On the other hand, if there is a decomposition of $\R^+$ into $n$ pieces, defined by some of the conditions \eqref{eqparts},
    then there exists a field $K$ which is finitely generated as a field over $\Q$ and the intersection of any of the $n$ pieces with $K^+$ is non-empty. Indeed, take an element from each piece of the decomposition of $\R^+$, denote them by $p_1, p_2, \dots, p_n$ and take the field $K= \Q(p_1, p_2, \dots, p_n)$, which is finitely generated. 
    
    Given that for any countable cardinal $k$ we constructed $k$-decompositions of a finitely generated field $K$ over $\Q$ of transcendence degree at least 2, it is natural to ask into what those decomposition pieces can look like.
    We believe that the pieces of each finite decomposition can be determined using a sequence of derivations. First of all, we conjecture that in the case of finite decompositions the pieces as $\Q$-convex sets are separated by separating hyperplanes, which are closed under multiplication (and addition) but we should say a bit more. The separating hyperplanes can also contain elements of the field so if we would like to properly describe the decompositions, then we have to specify what happens on these hyperplanes (whose existence we only conjecture). We believe that the partition on these hyperplanes is also determined by derivations defined on these one codimensional subspaces giving rise to $2$ codimensional separating hyperspaces.
    This process can be continued in order to obtain a proper description for finitely many pieces.
    
    \section{Elementary observations}\label{sec3}
    In this section we prove a technical lemma, which will be used during the process when we determine all $2$-decompositions of some extensions of the field of rational numbers.
    As a consequence of this lemma we immediately get that for every non-trivial decomposition into finitely many pieces, the pieces cannot all be measurable. 
    
    Let $K \subseteq \mathbb{R}$ be a field. Clearly, $K$ contains $\mathbb{Q}$ and let us assume that we have a decomposition $K^+=H_1 \sqcup H_2 \sqcup \ldots \sqcup H_{\lambda}$ into disjoint invariant pieces. 
    It is easy to see that each part of the decomposition is closed under multiplication by positive integers and therefore also under multiplication by positive rational numbers. The following lemma gives more information about the decomposition.
    \begin{lemma}\label{lemraceltol}
    \begin{enumerate}[(a)]
    \item Let us assume that $a, a+q \in H_1$, where $q \in \mathbb{Q}^{+}$. Then every element of the form $a+r$ with $r \in \mathbb{Q}^{+}$ is also in $H_1$.
    \item If $\lambda$ is finite, then for every $a \in K^+$ and $a+q>0$, where $q \in \mathbb{Q}$, we have that $a$ and $a+q$ are in the same part of the decomposition.
    \end{enumerate} 
    \end{lemma}
    \begin{proof}
    \begin{enumerate}[(a)]
    \item It follows from the previous observations that $H_1$ is closed under convex combinations with rational weights, which gives the result for $0<r<q$.
    It also shows that we only have to show that there exist a sequence of rational numbers $q_i$ tending to infinity such that $a+q_i$ are all in $H_1$.
    
    Now $(a+2r)^{2}=a^{2}+4 (a+r) r$, showing that $(a+2r)^2 \in H_1$ whenever $a$ and $a+r \in H_1$ as well. It follows that $a+nr$ is also in $H_1$ for every $n\in \mathbb{N}$.
    \item Easy consequence of the previous case and the pigeonhole principle.
    \end{enumerate}
    \end{proof}
    
    As a corollary we obtain the following. 
    
    \begin{cor}\label{cor1deg} For any finite decomposition of $\mathbb{Q}(\alpha)^{+}$, the polynomials of degree $1$ can only be cut into one or two pieces. If there are two pieces, then the membership is determined by the sign of the (leading) coefficient of the monomial of degree $1$.
    \end{cor}
    
    As another consequence of the lemma we prove the following result.
    \begin{thm}\label{thm:32}
    Let $\R^+ =H_1\sqcup  \ldots \sqcup  H_k$ be a decomposition  of $\R^+$ with $k \ge 2$. Then at least one of the pieces is not Lebesgue measurable. 
    \end{thm}
    \begin{proof}
    Let us assume that every piece is Lebesgue measurable. Then one of them has to be of positive measure.
    It follows from Lemma \ref{lemraceltol} that each piece is dense in $\R^{+}$ and for each piece $H_i$ we have $H_i+H_i=\left\{h+h' \mid h,h' \in H_i \right\} \subseteq H_i$. Steinhaus's theorem \cite{steinhaus} implies that if a set $H \subseteq \R$ is of positive measure, then $H+H$ contains an interval. Therefore $H_i$ contains an interval, which contradicts the fact that the other pieces are dense.
    \end{proof}
    \begin{observation}
    \begin{enumerate}
    \item 
    A similar proof would show that at least one of the pieces is not Borel using Piccard's theorem but this simply follows from the fact that one of the pieces is not Lebesgue measurable. 
    \item As a strengthening of Theorem \ref{thm:32} we obtain that if a set in the decomposition is measurable, then it is of measure zero. 
    \item Similarly, every Borel set appearing in the decomposition is of first category by Piccard's theorem, which says for a Borel set $A$ of second category we have that $A+A$ contains an interval, see \cite{kechris}.   
    \end{enumerate}
    \end{observation}
    Note that the proof of \cite[Lemma 2.2]{elek} applied for $\R^+$ would actually give a similar result.
    
    
    \section{One variable polynomial ring}\label{secdegree1}
    In this section we investigate the decompositions of the ring extension $\Q[\alpha]$, where $\alpha\in\R$ is transcendental. Our arguments will work in slightly greater generality so as to be applicable in the case of an algebraic $\alpha$ as well, and to make them clearer we introduce the following notation:
    \[
    \Q[x]^+_{x=\alpha} = \{p\in\Q[x] \mid p(\alpha) >0 \}.
    \]
    Note that in the case of transcendental $\alpha$ this is just a relabeling of $\Q[\alpha]^+$, and we will consider them to be identified in the natural way.
    
    We show that for any real $\alpha$ there are only two different decompositions of $\Q[x]^+_{x=\alpha}$ into $2$ nonempty pieces.
    \begin{prop}\label{propegyvaltpolinom}
    Let $\alpha \in \R$ and let $\Q[x]^{+}_{x=\alpha}
    =H_1 \sqcup  H_2$ be a decomposition into two invariant pieces. 
    Then
    \begin{equation*}
    \begin{aligned}
    H_1&\subset(H_+ \cup H_0) \text{ and } H_2\subset (H_- \cup H_0), &\phantom{=} \text{ or}\\
    H_1&\subset(H_- \cup H_0) \text{ and } H_2\subset(H_+\cup H_0), &\phantom{=} \phantom{\text{ or}}
    \end{aligned}
    \end{equation*}
    where
    \begin{equation*}
    \begin{split}
    H_+=\{ p\in \Q[x]^+_{x=\alpha} \mid p'(\alpha) >0 \},\\
    H_0=\{ p\in \Q[x]^+_{x=\alpha} \mid p'(\alpha) =0 \},\\
    H_-=\{ p\in \Q[x]^+_{x=\alpha} \mid p'(\alpha) <0 \}.\\
    \end{split}
    \end{equation*}
    \end{prop}
    In the case of transcendental $\alpha$ the pieces coincide with the ones listed in Theorem \ref{thm1valt}, since $H_0=\Q^+$ and Lemma \ref{lemraceltol} ensures that it must belong to either $H_1$ or $H_2$ in its entirety.
    
    We introduce some notation concerning the topology of $\Q[x]$ that we use in our investigation. 
    We denote by $\Q[x]^{<N}$ the set of polynomials over $\Q$ of degree smaller than $N$ and by $\Q[x]^{N}$ the set of polynomials over $\Q$ of  degree exactly $N$. 
    The  vector space $\Q[x]^{<N}$  is $N$ dimensional over $\Q$, and as a topological space can 
    naturally be identified with $\Q^N$, making it a metric space as well.
    Using the natural topology induced by the metric, it is clear that the addition is a continuous function from $\Q[x]^{<N} \times \Q[x]^{<N}$ to $\Q[x]^{<N}$. Similarly, the multiplication is a continuous map from $\Q[x]^{<N} \times \Q[x]^{<N}$ to $\Q[x]^{<2N-1}$.
    
    
    Using a suitable fixed ordering of the coordinates the elements of $\Q[x]$  can be identified with infinite sequences of rational numbers that contain only a finite number of nonzero coordinates. This space can be equipped with the distance $d\big((a_i),(b_i)\big)=\sum_{i=1}^{\infty} \frac{|a_i-b_i|}{2^i}$, making it a metric space. This is not a complete metric space but if we restrict our attention to the set of polynomials of degree at most $N< \infty$, then the metric is equivalent to the product metric on $\Q^{N+1}$ and hence the obtained topological space endowed with the subspace topology is homeomorphic to the natural topology defined above.
    
    
    
    Our strategy in the remainder of this section is to find and describe a hyperplane separating the two parts $H_1$ and $H_2$ of the decomposition of $\Q[x]^+_{x=\alpha}$. These are $\Q$-convex sets so the existence of such a hyperplane would be guaranteed in the finite dimensional vector spaces $\Q[x]^{<N}$: we embed $\Q[x]^{<N}\cong \Q^N$ into $\R[x]^{<N}\cong \R^N$ and the closures of $\Q$-convex sets are closed convex sets of $\R^N$ that might only intersect on their boundary. It is well known that in this case there is a nonzero vector $(\lambda_i)_{i=1}^N\in \R^N$ such that the scalar product of $(\lambda_i)$ with all the elements of $H_1$ are nonnegative and the scalar products with all the elements of $H_2$ are nonpositive. The following example shows that for infinite dimensional space this is not always the case in general, although in our case we will be able to conclude that. Thus before we start proving Proposition \ref{propegyvaltpolinom} we give an example of a pair of ($\Q$-)convex sets in an infinite dimensional vector space, which are not separated by a hyperplane (determined by a normal vector).
    
    In what follows $LC(f)$ denotes the leading coefficient of a one-variable polynomial $f$.
    \begin{ex}
    Let \begin{equation*}\begin{split}
    A&= \{f\in \Q[x]\colon LC(f)>0\},\\
    B&= \{f\in \Q[x]\colon LC(f)<0\}
    \end{split}
    \end{equation*}
     It is clear that sets $A$ and $B$ are closed under addition and both are $\Q$-convex. Furthermore, the set $A$ is also closed under multiplication. 
    
    Let us assume indirectly that there exist a nonzero vector $(\lambda_i)_{i\in \N} \in \R^{\N}$ such that the scalar product of  $(\lambda_i)$ with the elements of $A$ is nonnegative and the scalar product of the elements of $B$ with $(\lambda_i)$  is nonpositive. Since $x^i$ is contained in $A$ for every $i \ge 0$, we have that $\lambda_i \ge 0$ for every $i \ge 0$. If for some $i$ we had $\lambda_i>0$, then for all positive $\varepsilon$ the polynomial $-\varepsilon x^{i+1}+x^i$ would belong to $B$, but its scalar product with $(\lambda_i)$ would be $\lambda_i-\varepsilon\lambda_{i+1}$, which is positive for a sufficiently small positive $\varepsilon$ -- a contradiction. Hence all $\lambda_i=0$ in contradiction with the assumption of $(\lambda_i)$ not being the zero vector.
    \end{ex}
    
    Assume that for some $N$, none of the members of the following decomposition are empty: $\Q[x]^{<N}=(H_1 \cap \Q[x]^{<N}) \sqcup (H_2  \cap \Q[x]^{<N})$.
    Then $\Q[x]^{<N} \cap H_1$ and $\Q[x]^{<N} \cap H_2$ are separated by a hyperplane as we discussed above; by this, we mean that both pieces are contained in one of the two (closed) halfspaces determined by the hyperplane.
    Such a hyperplane is determined by one of its normal vectors $v_N$. If $M \ge N$, then we denote by $v_{M,N}$ the vector consisting of the first $N$ coordinates of $v_M$.
    Further we say that two vectors $u$ and $v$ are {\it equivalent} if $u$ is a nonzero multiple of $v$. In this case we write $u \sim v$.
    
    It is clear that we may only expect $v_N$ to be unique up to multiplication by nonzero real numbers. Also, it could happen that $(H_j \cap \Q[x]^{<N})$ is contained in a proper subspace of $\Q[x]^{<N}$. We would like to exclude these pathologies and we intend to show that $v_{M,N} \sim v_N$ if $M \ge N$ ($M,N \in \N$). First we show that if there is a decomposition of $\Q[x]^+_{x=\alpha}$ into two pieces, then the linear polynomials are already separated. 
    
    \begin{lemma}\label{lem:1foku}
    If all the linear polynomials in $\Q[x]^+_{x=\alpha}$ are contained in $H_j$, then $\Q[x]^{+}_{x=\alpha}=H_j$.  
    \end{lemma}
    
    \begin{proof} We show that if the linear polynomials are contained in $H_j$, then the positive rationals are also in $H_j$. Indeed, every element of $\Q^+$ can be written as the sum of linear polynomials taking positive values at $\alpha$. Thus $\Q[x]^1\cap \Q[x]^{+}_{x=\alpha} \subset H_j$ implies that $\Q[x]^{<2}\cap \Q[x]^{+} \subset H_j$. Without loss of generality we may assume $j=1$. 
    
    We prove by induction on $i$ that $\Q[x]^{<i}\cap \Q[x]^{+}_{x=\alpha}= \Q[x]^{<i} \cap H_1$ for every $i \ge 3$.
    Let $g \in \Q[x]$ be of degree $i-1$ with $g(\alpha)>0$. 
    Assume first that $LC(g)>0$. Then there exists a polynomial $h$ of degree $i-1$ written as the product of a positive constant and linear terms of the form $x-\beta$, where $\beta <\alpha$, such that $0<h(\alpha)<g(\alpha)$ and $LC(g)=LC(h)$. This can always be done by choosing these $\beta$'s close enough to $\alpha$. Then $(g-h)(\alpha)>0$  and $g-h$ is of degree at most $i-2$ so $g-h$ and $h$ are in $H_1$ by induction. Since $H_1$ is closed under multiplication and addition we have $g \in H_1$.
    
    A similar argument works if $LC(g)<0$ but in this case $h$ is the product of a positive constant and $i-2$ linear term of the form $x-\beta$ and one of the form $\beta'-x$, where $\beta < \alpha < \beta'$ is chosen to be close  enough that $(g-h)(\alpha)>0$ and $LC(g)=LC(h)$. Then similarly as above $g\in H_1$ follows.
    \end{proof}
    
    \begin{lemma}\label{lemnedfokuttartalmaz}
    At least one of $H_1$ and $H_2$ contains polynomials with positive and negative leading coefficient of degree $i$ for every $2 \le i$ and $H_j \cap \Q[x]^{N} \ne \emptyset$ for $j=1,2$.
    \end{lemma}
    \begin{proof}
    By Lemma \ref{lem:1foku} the linear polynomials are divided into two nonempty parts. Thus $H_1$ and $H_2$ both contain squares of polynomials of degree $1$, which have a positive leading coefficient. Since $H_1 \cup H_2$ is a partition $H_1$ or $H_2$ contains polynomials of negative leading coefficient of degree $2$ as well. 
    Thus the statement is true for $i=2$. 
    
    Now since both parts contain linear polynomials, multiplying by a suitable power of such a linear polynomial shows that for every $i \ge 2$ one of $H_1$ and $H_2$ contains polynomials of degree $i$ with positive and negative leading coefficients as well.
    
    The second statement directly follows from Lemma \ref{lem:1foku} by taking the $N$'th power of a polynomial of degree $1$.
    \end{proof}

    As a consequence of Lemma \ref{lemnedfokuttartalmaz} we obtain that for every $N \ge 2$ the part of $\Q[x]^{<N}$ that evaluates positively on $\alpha$ contains elements from both parts of the decomposition. Thus we obtain a decomposition of a halfspace (the polynomials that evaluate to positive numbers) of an $N$ dimensional vector space over $\Q$ into two convex subsets. Recall that $\Q[x]^{<N}\cong \Q^N$ can be embedded in $\R^N$ such that the closures of $\Q[x]^N \cap H_i$ ($i=1,2$) in $\R^N$ are convex sets intersecting only in a hyperplane. These type of convex subsets in finite dimensional vector space over $\R$ are separated by hyperplanes and in this case the hyperplane is uniquely determined by the pieces since they cover a halfspace, the positive part of $\Q[x]^{<N}$. Moreover, $v_N$, the normal vector of the hyperplane is not equal to $(0\stb 0,1)$ since at least one of the two parts of the decomposition contains both elements with positive and negative scalar product with $(0\stb 0,1)$ by Lemma \ref{lemnedfokuttartalmaz}.
    
    As a consequence of Lemma \ref{lem:1foku} we have that that $H_i \cap \Q[x]^{N}\ne \emptyset$ for all $N\in \mathbb{N}$ and $i=1,2$, and one of them, $H_1$, say, contains the elements of $\Q^+$. Since the complement of a hyperplane containing $\Q$ does not have a $\Q$-convex intersection with $\Q[x]^{<N} \cap \Q[x]^+_{x=\alpha}$, the set $H_1$ is not contained in any proper subspace of $\Q[x]^{<N}$ for every $N \ge 2$ . We can prove the same for $H_2$. 
    
    \begin{lemma}\label{lemproper}
    $H_i$ is not contained in any proper subspace of $\Q[x]^{<N}$ for every $N \ge 2$ and $i=1,2$.
    \end{lemma}
    \begin{proof}
    By Lemma \ref{lem:1foku}, the statement holds for $N=2$. Now we proceed by induction.
    Let us assume that $H_i \cap \Q[x]^{<N}$ is not contained in any proper subspace of $\Q[x]^{<N}$ but $H_i \cap \Q[x]^{<N+1}$ is contained in the subspace orthogonal to the vector $0 \ne w_{N+1} \in \R^{N+1}$. Now we consider the elements of $H_i \cap \Q[x]^{<N}$ as vectors in $\Q[x]^{<N+1}$ by appending a zero. Then, by the inductive hypothesis, for every vector $x\in \Q[x]^{<N+1}$ whose last coordinate is zero there is an element of $H_i$ that is not orthogonal to $x$. 
    Thus the first $N$ coordinates of $w_{N+1}$ have to be zero. Hence $w_{N+1}\sim(0,0 \stb 0,1)$, which contradicts the conclusion of Lemma \ref{lemnedfokuttartalmaz} that $H_i$ contains polynomials of degree $N$ with positive and negative leading coefficients.
    \end{proof}
    
    \begin{cor}\begin{enumerate}
    \item 
    There is a unique $v_N\nsim (0,\dots, 0,1) \in \R^N$ up to the equivalence $\sim$ such that the hyperplane orthogonal to $v_N$ separates $H_1 \cap \Q[x]^{<N}$ and $H_2 \cap \Q[x]^{<N}$. In particular, the closure of either $H_i$ in $\R^N$ has nonempty interior.
    \item 
    Let $N \ge 2$. The normal vectors are compatible among themselves, i.e. $v_{M,N} \sim v_N$ for every $M \ge N$. 
    \end{enumerate}
    \end{cor}
    \begin{proof}
    \begin{enumerate}
    \item This is a direct consequence of Lemma \ref{lemproper}.
    \item
    Let us assume indirectly that there are $M >K$ with $ v_{M,K} \not\sim v_K$. Then there is a $K \le N  < M$ with $v_{N+1,N} \not\sim v_N$. 
    As we have seen before $\Q[x]^{<N}$ naturally embeds into $\Q[x]^{<N+1}$.
    Now $v_N$ and $v_{N+1,N}$ are nonzero nonequivalent vectors since $v_{N+1} \ne (0 \stb 0,1)$, and the hyperplane orthogonal to the vector $v_{N+1,N}$ also separates $H_1$ from $H_2$. Consider now a vector $w$ that satisfies $\langle w,v_N \rangle > 0 > \langle w,v_{N+1,N}\rangle$; one of $w$ and $-w$ will lie in $\Q[x]^+_{x=\alpha}$ and thus either belong to both $H_1$ and $H_2$ or to closure of neither, which is a contradiction.
    \end{enumerate}
    \end{proof}
    
    As a corollary of this argument we might assume that $v_M$ is always an extension of $v_N$ if $M>N$. Thus we obtain an infinite sequence $(\lambda_i)\in \R^{\N}$ such that the restriction of this sequence to its first $N$ coordinates is equivalent to $v_N$ for every $N \ge 2$. 
    
    \begin{observation}
    As $\Q[x]^{<N}\cong \Q^N$ naturally embeds to $\R[x]^{<N}$, the polynomial ring $\Q[x]$ naturally embeds to $\R[x]$, an infinite dimensional vector space over $\R$. Further denote
    \[
    \R[x]^+_{x=\alpha}=\{p(x)\in \R[x]\mid p(\alpha)>0\}.
    \]
    It follows from the previous argument that $H_1$ and $H_2$ are separated by the hyperplane of the infinite dimensional space of $\R[x]$ orthogonal to the vector $(\lambda_i)_{i\in \N}$. Therefore we can take the closure of $H_1$ and $H_2$ in $\R[x]$ that is also separated by this hyperplane. This observation is critically used in the sequel.
    \end{observation}
     
    In our argument we only use hyperplanes that are defined by a vector $(\lambda_i)_{i\in \N}$ as follows: $\mathcal{H}=\{\sum a_i x^i \mid \sum a_i \lambda_i =0  \}$. From now on, we implicitly assume this for all of our hyperplanes.
    
    Above we showed that the invariant sets of a 2-decomposition of $\Q[x]^+_{x=\alpha}$ can be separated by a hyperplane of this type. 
    It will be called a \emph{separating hyperplane} of the decomposition and we denote it by $\mathcal{H}$. Next we describe the possible separating hyperplanes of $\Q[x]^+_{x=\alpha}$.
    
    Note that the separating hyperplane lives in $\R[x]$, and our condition of the pieces $H_j$ being closed under addition and multiplication is in (a subset of) $\Q[x]$. The points of the hyperplane in $\R[x]$ are typically (but not necessarily) not contained in $\Q[x]$. However, the usual addition and multiplication in $\Q[x]$ extend to those of $\R[x]$ in a continuous way, and we note that the two closed halfspaces defined by the separating hyperplane intersect $\R[x]^+$ in pieces that are also closed under addition and multiplication. Indeed, these two halfspaces contain $H_1$ and $H_2$, respectively, as a dense subset.
     
     \begin{lemma}\label{lemlimit}
     Let $\Q[x]^+_{x=\alpha}=H_1\sqcup H_2$ be a decomposition into $\Q$-convex subsets. Then for the separating hyperplane $\mathcal H$ the set $\mathcal H\cap \R[x]^+_{x=\alpha}$ consists of those elements of $\R[x]^+_{x=\alpha}$ that are limits of sequences both from $H_1$ and from $H_2$. Additionally, these sequences can be chosen be consist of polynomials of at most the same degree as their limit point.
    \end{lemma}
    
    \begin{proof}
    Take any $N\in \N$. In $\R[x]^{<N}$, the set $\mathcal{H} \cap \R[x]^{<N}$ is the common boundary of its two halfspaces, so only elements of $\mathcal H$ can be limits of sequences in both $H_1$ and $H_2$. On the other hand, as  $H_1 \sqcup H_2=\Q[x]^+_{x=\alpha}$, $H_1$ and $H_2$ are dense in the intersections of corresponding halfspaces with $\R[x]^+_{x=\alpha}$. Take an arbitrary positive element $h \in \mathcal{H}\cap \R[x]^{<N}$. Since $\Q[x]^{<N}$ is dense in $\R[x]^{<N}$, there exists a sequence $(h_n)_{n\in \N}$ in $\Q[x]^{<N}\cap H_1$ that converges to $h$ in $\R[x]^{<N}$. The same argument can be applied to a sequence of elements $H_2\cap \Q[x]^{<N}$ tending to $h$ as well.
    \end{proof}

    \begin{lemma}\label{lemmultadd}
    Let $\mathcal{H}\subset\R[x]$ be a separating hyperplane for a $2$-decomposition of $\Q[x]^+_{x=\alpha}$. Then $\mathcal{H} \cap \R[x]^+_{x=\alpha}$ is closed under addition and multiplication (i.e., invariant).
    \end{lemma}
    \begin{proof}
    Since $\mathcal{H}$ is a linear subspace and $\R[x]^+_{x=\alpha}$ is closed under both addition and multiplication, $\mathcal{H} \cap \R[x]^+_{x=\alpha}$ is closed under addition.
    
    Now we prove that $\mathcal{H}\cap \R[x]^+_{x=\alpha}$ is closed under multiplication as well. Namely, we show that for all $N$, if we multiply two elements of $\mathcal{H} \cap \R[x]^{<N} \cap \R[x]^+_{x=\alpha}$, then the product lies in $\mathcal{H} \cap \R[x]^{<2N-1} \cap \R[x]^+_{x=\alpha}$. By Lemma \ref{lemnedfokuttartalmaz} we have that if $H_1$ and $H_2$ cut $\Q[x]^+_{x=\alpha}$ into two pieces, then  both pieces intersect $\Q[x]^{<N}$ non-trivially for all $N\ge 2$.
    
    Now we take two elements $h,k \in \mathcal{H}\cap \R[x]^{<N} \cap \R[x]^+_{x=\alpha}$. By Lemma \ref{lemlimit} 
    there exist sequences $(h_{i,n})_{n\in \N}$ and $(k_{i,n})_{n\in\N}$ in $\Q[x]^{<N}\cap H_i$ ($i=1,2$) that are converging in $\R[x]^{<N}$ to $h$ and $k$, respectively. Since $h_{i,n} k_{i,n}$ are in $H_i$, Lemma \ref{lemlimit} implies that $hk\in\mathcal{H}$. 
    Moreover it is clear that $hk\in \R[x]^+_{x=\alpha}$ as $hk$ is a limit of polynomials that evaluate positively on $\alpha$ and $hk(\alpha)\ne 0$, finishing the proof of the lemma.
    \end{proof}

    \begin{cor}
    $\Q \subset \mathcal{H}$.
    \end{cor}
    \begin{proof}
    Every element of $\Q$ can be approximated by $r_1\alpha-r_2$ and $r'_1-r'_2\alpha$, which are in different parts by Corollary \ref{cor1deg}. Hence by the previous argument $\Q\in \mathcal{H}$.
    \end{proof}
    
    \begin{cor}\label{corwhole}
    The hyperplane $\mathcal H$ itself is closed under multiplication.
    \end{cor}
    
    \begin{proof}
    Lemma \ref{lemmultadd} tells us that $\mathcal H \cap \R[x]^+_{x=\alpha}$ is invariant, so its defining coefficients $(\lambda_i)$ satisfy the corresponding algebraic condition
    \begin{equation}\label{eqinvariant}
    \begin{aligned}
    \forall (a_i),(b_i): \sum_i \lambda_i a_i=0,\sum_i a_i \alpha^i >0,\sum_i \lambda_i b_i=0,\sum_i b_i \alpha^i >0 \Longrightarrow\\
    \Longrightarrow\sum_i \lambda_i \sum_{0 \le j \le i} a_j b_{i-j} =0.
    \end{aligned}
    \end{equation}
    The equalities $\sum_i \lambda_i x_i=0,\sum_i \lambda_i y_i=0$ define a closed plane in each $\R[x]^{<N} \times \R[x]^{<N}$, and on it the two additional inequalities cut out a relative open set. Since the conclusion holds on this open set, it must also hold on its Zariski closure, which is the whole plane. That is, the condition \eqref{eqinvariant} holds true without requiring the inequalities and $\mathcal H$ is indeed closed under multiplication.
    \end{proof}
    
    \bigskip
    
    Our next step in describing the hyperplanes that can serve as a separating hyperplane is to determine which are closed under multiplication.
    
    
    Recall that the hyperplane $\mathcal{H}$ in $\R[x]$ is uniquely (up to scaling) determined by a sequence $(\lambda_i)_{i \in \mathbb{N}}$ of real numbers. A polynomial $p(x)=\sum_{i=0}^{\deg(p)} \alpha_i x^i$ is in $\mathcal{H}$ if and only if $\sum_{i=0}^{\deg(p)} \alpha_i \lambda_i =0$.  
    
    \begin{lemma}\label{leml1} $\lambda_1\ne 0$.
    \end{lemma}
    \begin{proof}
    By the way of contradiction, let us suppose that $\lambda_1=0$. Then $x \in \mathcal{H}$. Since $\mathcal{H}$ is closed under multiplication and addition by Lemma \ref{lemmultadd} we have that all polynomials 
    whose constant coefficient is $0$ are contained in $\mathcal{H}$. Also by Lemma \ref{lemraceltol} we have that each piece of the decomposition is closed under translation by positive rational numbers, hence so is $\mathcal{H}$. This shows that $\mathcal{H}$ contains an open set in $\R[x]$, contradicting the fact that $\mathcal{H}$ is a separating hyperplane. 
    \end{proof}
    
    As a special case we obtain again that the (positive) linear polynomials are cut into two nonempty pieces by $H_1$ and $H_2$.
    
    It is clear that $\lambda_i x-\lambda_1 x^{i}$ and $\lambda_j x-\lambda_1 x^{j}$ belong to $\mathcal{H}$. By Lemma \ref{lemmultadd}, the product of two such polynomials also belongs to $\mathcal{H}$. Thus 
    $$
    \lambda_i \lambda_j x^2 -\lambda_i\lambda_1 x^{j+1} - \lambda_j \lambda_1 x^{i+1} +\lambda_1^2 x^{i+j} \in \mathcal{H},
    $$
    which is equivalent to the fact that 
    $$
    \lambda_i \lambda_j \lambda_2 -\lambda_i\lambda_1 \lambda_{j+1} - \lambda_j \lambda_1 \lambda_{i+1} +\lambda_1^2 \lambda_{i+j}=0.
    $$
    In particular, choosing $i=2$ gives the following recursion for the coefficients:
    \begin{equation*} 
    \lambda_2^2 \lambda_j - \lambda_1 \lambda_2 \lambda_{j+1} - \lambda_1 \lambda_3 \lambda_j + \lambda_1^2 \lambda_{j+2} = 0.
    \end{equation*}
    
     This recursion for the coefficients shows that $\lambda_1$, $\lambda_2$ and $\lambda_3$ determine the hyperplane $\mathcal{H}$.
    Since the sequences $(\lambda_i)$ and $(\delta \lambda_i)$ determine the same hyperplane for every $\delta \in \R\setminus \{ 0\}$ we may assume that $\lambda_1=1$.
    
    \begin{lemma}\label{lemelvalasztohipersik}
    The following are the only hyperplanes in $\R[x]$ with $\lambda_0=0$ and $\lambda_1\neq 0$ that are closed under multiplication:
    \begin{enumerate}
    \item $\mathcal{H}_{\beta,\gamma} = \{ p \in \Q[x] : p(\beta) = p(\gamma) \}$ for some real numbers $\beta \ne \gamma$, 
    \item $\mathcal{H}_{\delta }= \{ p \in \Q[x] : p'(\delta) = 0 \}$ for some $\delta \in \R$,
    \item $\mathcal{H}_{z,\overline z} = \{ p \in \Q[x] : p(z) = p(\overline z) \} = \{ p \in \R[x] : \operatorname{Im} p(z) = 0\}$ for some $z=u+iv \in \C\setminus \R$.
    \end{enumerate}
    \end{lemma}
    \begin{proof}
    It is readily verified that these hyperplanes are closed under multiplication. The argument above implies that the pair $(\lambda_2, \lambda_3)$ (under the assumption $\lambda_1=1$) determines such a hyperplane. We will show that all values $(\lambda_2,\lambda_3) \in \R^2$ occur in one of the listed cases, and this will finish the proof of Lemma \ref{lemelvalasztohipersik}. To calculate $\lambda_i$, we will use the polynomial $x^i-\lambda_i x$ that has to belong to $\mathcal H$.
    
    
    In the first case ($\mathcal H_{\beta,\gamma}$) we get the condition $\gamma^i-\lambda_i \gamma=\beta^i-\lambda_i \beta$, which yields $\lambda_i=\frac{\beta^{i}-\gamma^{i}}{\beta-\gamma}$ and in particular
    \begin{align*}
    \lambda_2&=\beta+\gamma,\\
    \lambda_3&=\beta^2+\beta\gamma+\gamma^2.
    \end{align*}
    It is straightforward to see that these value pairs cover the region $\lambda_2^2<\frac{4}{3}\lambda_3$ of the parameter space.
    
    In the second case  ($\mathcal H_{\delta}$) we get the condition $i\delta^{i-1}-\lambda_i=0$, in particular
    \begin{align*}
    \lambda_2&=2\delta, \\
    \lambda_3&=3\delta^2,
    \end{align*}
    and these cover the parameter pairs that satisfy $\lambda_2^2=\frac{4}{3}\lambda_3$.
    
    Finally, in the third case ($\mathcal H_{z,\overline z}$) we get
    \begin{align*}
    \lambda_2 &=  \frac{\operatorname{Im} z^2}{\operatorname{Im} z}=2u,\\
    \lambda_3&=\frac{\operatorname{Im} z^3}{\operatorname{Im} z}=3u^2-v^2,
    \end{align*}
    covering the rest of the parameter space ($\lambda_2^2 > \frac{4}{3}\lambda_3$).
    
    \end{proof}
    All of these hyperspaces naturally separate two classes of polynomials determining the possible decompositions of $\Q[x]$:
    \begin{itemize}
    \item $\mathcal{H}_{\beta,\gamma}$ splits $\Q[x]$ into 
    \begin{align*}
        S_{\beta,\gamma}&=\left\{ p \in \Q[x] : p(\beta) > p(\gamma) \right\} \text{ and}\\
        S_{\gamma,\beta}&=\left\{ p \in \Q[x] : p(\gamma) > p(\beta) \right\};
    \end{align*}
    \item $\mathcal{H}_{z,\overline z}$ splits $\Q[x]$ into 
    \begin{align*}
        S_{z,\overline z}&=\left\{ p \in \Q[x] : \operatorname{Im} p(z) = - \operatorname{Im} p(\overline z) > 0 \right\}\text{ and} \\
        S_{\overline z,z}&=\left\{ p \in \Q[x] : \operatorname{Im} p(z) = - \operatorname{Im} p(\overline z) < 0 \right\};
    \end{align*}
    \item $\mathcal{H}_{\delta}$ splits $\Q[x]$ into
    \begin{align*}
        S^+_{\delta}&=\left\{ p \in \Q[x] : p'(\delta) > 0 \right\}\text{ and} \\
        S^-_{\delta}&=\left\{ p \in \Q[x] : p'(\delta) < 0 \right\}.
    \end{align*}
    
    \end{itemize}
    It remains to check which of these classes are closed under multiplication in $\Q[x]^+_{x=\alpha}$. 
    We claim that the separating hyperspace for a $2$-decomposition of $\Q[x]^+_{x=\alpha}$ can only be $H_{\alpha}$ and it follows from the following lemma.
    \begin{lemma}\label{lemnemzart}\begin{enumerate}
    \item 
     For every  $\beta, \gamma \in \R \setminus \{ \alpha \}$ such that $\beta \ne \gamma$ the open subspaces
    $S_{\beta,\gamma}\cap \Q[x]^+_{x=\alpha}$ are not closed under multiplication. Further, if $\beta \ne \alpha$, then $S_{\alpha,\beta}\cap \Q[x]^+_{x=\alpha}$ is not closed under multiplication.
    \item For every $z \in \C \setminus \R$ the open subspaces $S_{z,\overline z}\cap \Q[x]^+_{x=\alpha}$ are not closed under multiplication.
    \item For every $\delta\ne \alpha\in \R$ the open subspaces $S_{\delta}^{+}\cap \Q[x]^+_{x=\alpha}$ and $S_{\delta}^{-}\cap \Q[x]^+_{x=\alpha}$ are not closed under multiplication.
    \end{enumerate}
    \end{lemma}
    \begin{proof}
    \begin{enumerate}
    \item 
    If $\alpha$, $\beta$ and $\gamma$ are 3 different real numbers, then there exists $p \in \Q[x]^+_{x=\alpha}$ such that $p(\gamma) >0$ and $p(\beta) <- p(\gamma)$. In this case $p^2(\beta) >p^2(\gamma)$ so $S_{\gamma, \beta}\cap\Q[x]^+_{x=\alpha}$ is not closed under multiplication. 
    The same argument works for $S_{\alpha,\beta}$ if $\beta \ne \alpha$.
    \item It can easily be seen that if $p \in \Q[x]^+_{x=\alpha}$ with $p(z) \notin \R$ and $\operatorname{Im}(p(z))>0$, then for a suitable $n \in \N$ we have $\operatorname{Im}(p(z)^n)<0$. This shows that  $S_{z,\overline z}\cap \Q[x]^+_{x=\alpha}$ is not closed under multiplication.
    \item If $\beta \neq \alpha$, then there exist $p_+,p_- \in \Q[x]^+_{x=\alpha}$ such that $p_\pm(\beta)<0$, while $p_+'(\beta) > 0$ and $p_-'(\beta) <0$. Taking this $p_\pm$ one gets $p_\pm \in S^\pm_{\delta} \cap \Q[x]^+_{x=\alpha}$, but $p_\pm^2 \in S^\mp_{\delta} \cap \Q[x]^+_{x=\alpha}$, so neither of $S^\pm_{\delta} \cap \Q[x]^+_{x=\alpha}$ is closed under multiplication.
    \end{enumerate}
    \end{proof}
    
    \noindent
    \textit{Proof of  Proposition \ref{propegyvaltpolinom}.} 
    It follows from Lemma \ref{lemnemzart} that $\mathcal{H}_{\alpha}$ is the only possible separating hyperplane. 
    Therefore $H_1$ and $H_2$ must each contain one the open halfspaces $H_+$ and $H_-$ defined by $\mathcal H_\alpha$, with $H_0$ covering the remaining elements of $\Q[x]^+_{x=\alpha}$.
    \qed
    
    \section{Lifting the decomposition}\label{sec5}
    In this section we show how a 2-decomposition of $\Q[\alpha]^{+}$ for a transcendental $\alpha$ extends to $\Q(\alpha)^{+}$. As we described the possible 2-decompositions of $\Q[\alpha]^{+}$, this will imply the characterisation of 2-decompositions stated in Theorem \ref{thm1valt}.
    
    \medskip
    \noindent
    {\it Proof of Theorem \ref{thm1valt}.}
    Suppose that $\Q(\alpha)^{+} =H_1 \sqcup H_2$ is a non-trivial decomposition of $\Q(\alpha)^{+}$ that is closed under addition and multiplication. This decomposition restricts to $\Q[\alpha]^+ = (H_1 \cap \Q[\alpha]^+) \sqcup (H_2 \cap \Q[\alpha]^+)$. One of these sets may be empty but we will first assume that this is not the case and conclude that then this splitting uniquely determines the decomposition of $\Q(\alpha)^+$ as described in \eqref{eqparts}. 
    
    By Proposition \ref{propegyvaltpolinom} we have two choices for 
    $H_1 \cap \Q[\alpha]^+$ and $H_2 \cap \Q[\alpha]^+$. Without loss of generality we can assume that $H_+\subset H_1$ (i.e., $H_1$ contains all polynomials that have positive derivative at $\alpha$) and $H_{-}\subset H_2$ (i.e., $H_2$ contains all polynomials that have negative derivative at $\alpha$). The constant rational polynomials can be included in $H_1$ or $H_2$ arbitrarily, 
    but two nonconstant polynomials whose derivatives have different signs at $\alpha$ must lie in different classes of the decomposition. 
    
    We would like to decide which set contains $\frac{p_1}{p_2}(\alpha)$ for $p_1, p_2\in \Q[x]^+_{x=\alpha}$ coprime polynomials, with $p_2$ nonconstant. If $p_1(\alpha)\in H_1$ and $p_2(\alpha)\in H_2$, then $\frac{p_1}{p_2}(\alpha)\in H_1$, otherwise $p_2\cdot \frac{p_1}{p_2}=p_2$ would imply that $p_1(\alpha)\in H_2$, a contradiction. The same argument works with the role of $H_1$ and $H_2$ reversed, therefore, from now on we assume that the sign of $p_1'(\alpha)$ and $p_2'(\alpha)$ is the same, and $p_1$ is not a (non-zero) rational multiple of $p_2$.
    Since $\alpha$ is transcendental we have $p_1(\alpha) \ne p_2(\alpha)$. We show that there exists a linear polynomial $l(x)=ax+b\in \Q[x]^{+}_{x=\alpha}$ such that $(p_1l)'(\alpha)$ and $(p_2l)'(\alpha)$ have different signs. Indeed, $(p_i(x)(ax+b))'=p_i'(x)(ax+b)+a p_i(x)$. Such an expression is positive at $\alpha$ if and only if $\frac{p_i'(\alpha)}{p_i(\alpha)} >\frac{-a}{a \alpha+b}$. Note that in our case  $a \alpha+b$ has to be positive. 
    
    We claim that $\frac{p_1'(\alpha)}{p_1(\alpha)} \ne \frac{p_2'(\alpha)}{p_2(\alpha)}$. The equality would imply $(p_1'p_2-p_1p_2')(\alpha)=0$, which would in turn give $\left(\frac{p_1}{p_2}\right)'(\alpha)=0$. 
    
    
    Now it remains to find $a$ and $b$ in $\Q$ such that $\frac{-a}{a \alpha +b}$ lies between 
    $c:=\left(\frac{p_1'}{p_1}\right)(\alpha)$ and $d:=\left(\frac{p_2'}{p_2}\right)(\alpha)$. This can be done by choosing $a$ to be a  rational number between $c$ and $d$ and then choosing $b \in \Q$ in a way that $a \alpha +b$ is suitably close to $1$. 
    
    To summarise, the algorithm that decides which of $H_1$ or $H_2$ contains $t =\frac{p_1}{p_2}\in \Q(x)^+_{x=\alpha}$ (with $p_1,p_2 \in \Q[x]^+_{x=\alpha}$) is as follows. If $p_2$ is constant or $p_1$ and $p_2$ are in different classes of the decomposition, then $t$ is in the same class as $p_1$. Otherwise, it follows from the previous argument that $t(x)$ can also be written as $\frac{p_1(x)l(x)}{p_2(x)l(x)}$, where $p_1(x)l(x)$ and $p_2(x)l(x)$ are not in the same class. This assign a membership to $t$ in $H_1$ or $H_2$. Note that the decomposition extends uniquely from $\Q[\alpha]^+$ to $\Q(\alpha)^+$. Thus, if $\Q[\alpha]^+$ is decomposed into two pieces non-trivially, then the only possible decomposition of $\Q(\alpha)^+$ is the one given in Theorem \ref{thm1valt}.

    
    To finish the proof of Theorem \ref{thm1valt} it remains to prove that if there is a 2-decomposition of $\Q(\alpha)^+$, then $\Q[\alpha]^+$ is also decomposed into two parts non-trivially. Indirectly, suppose that $\Q[\alpha]^+$ is contained in one of the invariant sets, say $\Q[\alpha]^+\subset H_1$. This also implies that $\Q\subset H_1$. We show that in this case $\Q(\alpha)^+\subset H_1$. For this it is enough to show that for every polynomial $q\in \Q[x]^+_{x=\alpha}$ the rational function $\frac{1}{q}(\alpha)$ also belongs to $H_1$. Indeed, if it is the case, then all $\frac{p}{q}(\alpha)\in H_1$ as the product of $p(\alpha) \in H_1$ and $\frac{1}{q}(\alpha)\in H_1$ and $H_1$ is closed under multiplication. Therefore, we indirectly assume that there is a $q \in \Q[x]^+_{x=\alpha}$ such that $\frac{1}{q}(\alpha) \in H_2$. Now we take $\Q\left(\frac{1}{q}\right)^+_{\frac{1}{q}=\frac{1}{q(\alpha)}}$, which is the set of the positive elements of a field of transcendence degree 1. Since $\Q\subset H_1$ and $\frac{1}{q}(\alpha)\in H_2$, $\Q(\frac{1}{q})^+_{\frac{1}{q}=\frac{1}{q(\alpha)}}$ is cut into two pieces non-trivially, hence its decomposition is determined by 
    the sign of the derivative with respect to $\frac{1}{q}$ evaluated at $\alpha$.
    
    In particular, as $\Q[q]^+_{q=q(\alpha)}\subset \Q(\frac{1}{q})^+_{\frac{1}{q}=\frac{1}{q(\alpha)}}$, $\Q[q]^+_{q=q(\alpha)}$ is also decomposed into two parts non-trivially. Indeed, 
    take any rational number $r>q(\alpha)$. The numbers $q(\alpha)$ and $r-q(\alpha)$ are both positive and the derivatives (with respect to $\frac1q$) of $q(x)$ and $r-q(x)$ as rational functions of $\frac{1}{q}$ have opposite signs at $\alpha$. Hence $q(\alpha)$ and $r-q(\alpha)$ belong to different $H_i$'s, in contradiction to our initial indirect assumption.
    
    Thus we get that if there is a 2-decomposition of $\Q(\alpha)^+$, then $\Q[\alpha]^+$ is also decomposed into two parts non-trivially. This finishes the proof of Theorem \ref{thm1valt}.
    \qed
    
    \begin{remark}
    As we have seen earlier each element of $\Q(x)^+_{x=\alpha}\setminus \Q[x]^+_{x=\alpha}$ can be written as $\frac{p}{q}$, where $p(\alpha)$ and $q(\alpha)$ are positive. Also, it is clear that such a quotient is not uniquely written so $p'(\alpha)$ and $q'(\alpha)$ have different sign.
    Now having positive derivative means that the corresponding real valued function is locally increasing at $\alpha$ and function with negative derivative at $\alpha$ locally decrease at $\alpha$. Having assumed that all these functions are positive at $\alpha$ we obtain that a decreasing decreasing function divided by one that is increasing is decreasing and vice-versa. This confirms that the extension preserves the property that the invariant sets are determined by a derivation.
    \end{remark}
    
    \section{Algebraic extension}\label{sec6}
    In this section we prove that there is no nontrivial decomposition of an algebraic extension of $\Q$ into 2 pieces. Now we recall Theorem \ref{thmalgebraic}. 
    
    \medskip
    \noindent
    {\bf Theorem 1.2.}
    Let $a$ be an algebraic element over $\Q$. Then $\Q(a)^{+}$ is not $2$-decomposable.
    
    \begin{proof}
    Let $F=\Q(a)$, where $a$ is an algebraic number and let $n$ be the degree of the extension $[F:\Q]$. $\Q(a)$ is an $n$ dimensional vector space over $\Q$, and we consider it embedded into $\R^n = \Q^n \otimes \R$ in the natural way. The multiplication on $\Q(a)$ also induces a multiplication on $\R^n$, making it a commutative ring. Any partition of $\Q(a)^+$ into $\Q$-convex sets gives rise to a covering of a closed halfspace of $\R^n$ by the closures of the partition classes.
    Since our space is finite-dimensional, these closures are separated by a hyperplane $H$.
    
    Consider now the ring homomorphism $\Phi: \R[x] \to \R^n$,
    \[
    \Phi\left(\sum c_i x^i\right)=\sum {a^i \otimes c_i} \in \Q(a) \otimes \R,
    \]
    Its restriction to $\Q[x]$ is surjective and sends $p$ to $p(a)$, so taking the preimages of the partition classes on $\Q(a)^+$ gives a partition of $\Q[x]^+_{x=a}$ into 2 nonempty invariant sets. These sets are separated by the preimage $\Phi^{-1}(H)$ of the hyperplane $H$, so this preimage is a separating hyperplane and contains the kernel of $\Phi$.
    
    \begin{lemma}
    The kernel of $\Phi$ is the ideal generated by the minimal polynomial of $a$.
    \end{lemma}
    \begin{proof}
    The multiples of the minimal polynomial of $a$ are clearly in the kernel of $\Phi$, and form a set of codimension $n$.
    \end{proof}
    
    \noindent
    By Proposition \ref{propegyvaltpolinom}, the separating hyperplane can only be $\mathcal H_a$ (notation of Lemma \ref{lemelvalasztohipersik}), but the minimal polynomial of $a$ is in the kernel of $\Phi$ but not in $\mathcal H_a$ - a contradiction.
    \end{proof}

    \begin{cor}\label{coralgebraic}
    Let $K$ be an algebraic field extension of $\Q$. Then there is no $2$-decomposition of $K$. 
    \end{cor}
    \proof Let $K$ be an algebraic field over $\Q$. Each non-trivial decomposition of $K^+$ gives a non-trivial decomposition of some subfield of $K$ that is a finite extension of $\Q$. Since $\Q$ is a perfect field every finite extension of $\Q$ is simple. Therefore it is enough to show that there is no $2$-decomposition of $F=\Q(a)$ and we are done by Theorem \ref{thmalgebraic}.
    \qed
    \section{Further directions}\label{sec7}
    There are several different questions arising. We have determined the possible 2-decompositions of a degree 1 transcendental extension of $\Q$. There are several obstacles on going one step further in many directions. 
    \begin{itemize}
    \item 
    One option would be to try to determine all $k$-decompositions of $\Q(\alpha)^+$ for every positive integer $k$. The main obstacle here is to prove that the parts of the decomposition are separated by hyperplanes (which are closed under multiplication and are defined by a nonvanishing sequence of coefficients). If it were true we could use Lemma \ref{lemelvalasztohipersik} and a suitable adaptation of Lemma \ref{lemnemzart} to determine all decompositions into finitely many pieces. Thus we conjecture that $\Q(\alpha)^+$ does not have proper $k$-decompositions if $k>3$ and there is a unique $3$-decomposition, which is the following: 
    \[
    \Q(\alpha)^+ =H_+ \sqcup H_0 \sqcup H_- .
    \]
    
    \item 
    It is rather an interesting open question whether there exists a decomposition of $\Q(\alpha)^+$ or at least $\Q[\alpha]^+$ into (countably) infinitely many sets that are closed under addition and multiplication. Such a construction could be given by taking each $l_r(\alpha)\in \Q[\alpha]^+$ into different sets for different $r\in \Q$, where $l_r(x)$ are linear functions of the form $x-r$ or $r-x$ (depending on the positivity of $l_r(\alpha)$). This would imply dissection of $\Q[\alpha]^+$ by rings generated by linear functions. It can be seen that the elements of these rings are disjoint. On the other hand, it is easy to show that not all quadratic functions are covered in this way so we should extend this construction for polynomials of higher degree. We note that it is not trivial to decide whether this construction can be extended to $\Q[\alpha]^+$ and further to $\Q(\alpha)^+$. The problem  of finding disjoint invariant sets leads to seeking the solutions of equations of type $P\circ Q(x)= R\circ S(x)$, where $P, Q\in \Q^+[x]$ are polynomials and $R,S \in \Q^+(x)$ are rational functions. Even if $P, Q\in \Q^+[x]$ are also polynomials, the answer to this problem is not clear. 
    
    Note that for general $P,Q,R,S\in \Q[x]$ this problem was first investigated and solved by J. F. Ritt \cite{Ritt}. On the other hand, the general solution, if $P,Q,R,S\in \Q(x)$ are rational functions is not known. The best of our knowledge is based on the work of F. Pakovich \cite{Fed}, therefore the interested reader is referred to that and the reference therein.
    
    Concerning our investigation we ask whether this construction can work at least for $\Q[\alpha]^+$ and if so we conjecture that there is a decomposition of $\Q(\alpha)^+$ into countably infinitely many invariant sets. 
    
    \item
    It would be nice to extend our result to transcendental extensions of degree at least 2. The potential separating hyperplanes, which are closed under multiplication, can be described as in Section \ref{secdegree1} and at least if the transcendence degree of the extension over $\Q$ is 2 we obtain that the possible separating hyperplanes are determined by derivations. (This has been verified by the authors but since this has no immediate consequence, we do not include the proof here.) However, the existence of separating hyperplanes is not yet clear even in the case of $2$-decompositions. Moreover, generally it seems to be a challenging problem to prove or disprove  whether any $2$-decomposition of $\R$ is determined by derivations (or by the concept of separating hyperplanes).

    \item It is also an interesting question how we can determine the finite decompositions of $\Q(\alpha_1, \dots, \alpha_m)^+$ into invariant sets for algebraically independent elements $\alpha_1, \dots, \alpha_m\in \R$ over $\Q$. 
    To illustrate a difficulty that arises we take the field $K=\Q(\alpha_1,\alpha_2)$, where $\alpha_1,\alpha_2$ are algebraically independent elements, and we try to find all 2-decomposition of $K^+$. One may try to prove that all such decompositions come from taking a derivation of the form $d=c_1 \frac{\partial}{\partial \alpha_1}+c_2\frac{\partial}{\partial \alpha_2}$, where $c_1, c_2\in \R$ are arbitrary. Then we can find two different 2-decompositions of $K$ similar to those in Theorem \ref{thm1valt} such that 
    \begin{equation}\label{eqdersplit}
    \begin{split}
        H_1&=\{r\in K \colon d(r)\ge 0\}, H_2=\{r\in K \colon d(r)< 0\} \\
        H'_1&=\{r\in K \colon d(r)> 0\}, H'_2=\{r\in K \colon d(r)\le 0\}. 
    \end{split}
    \end{equation}
    But in fact, there are other 2-decompositions of $K$. For example, we take the sets 
    \begin{align*}
    H^1_+&=\left\{r\colon\frac{\partial r(\alpha_1, \alpha_2)}{\partial \alpha_1}>0\right\},\\
    H^1_-&=\left\{r\colon\frac{\partial r(\alpha_1, \alpha_2)}{\partial \alpha_1}<0\right\},\\
    H^1_0&=\left\{r\colon\frac{\partial r(\alpha_1, \alpha_2)}{\partial \alpha_1}=0\right\}.
    \end{align*}
    Then we take $H^1_0$ and divide it into 2 invariant set as follows: 
    \begin{align*}
    H^{1,2}_+&=H^1_0\cap \left\{r\colon\frac{\partial r(\alpha_1, \alpha_2)}{\partial \alpha_2}>0\right\},\\
    H^{1,2}_{0,-}&=H^1_0\cap \left\{r\colon\frac{\partial r(\alpha_1, \alpha_2)}{\partial \alpha_2}\le 0\right\}.
    \end{align*}
    Now set $H_1=H^1_+\cup H^{1,2}_+$ and $H_2=H^1_+\cup H^{1,2}_{0,-}$. It is straightforward to verify that this gives a 2-decomposition of $K$. Moreover, it is easy to show that $H_1$ 
    cannot be a set of the form \eqref{eqdersplit} defined by a derivation $d$ that is a linear combination of the partial derivatives.
    Note however that this example is still based on the idea of using separating hyperplanes.

    \end{itemize}
    \section*{Acknowledgement}
    G. Kiss was supported by the J\'anos Bolyai Fellowship of the Hungaria Academy of Sciences, the New National Excellence Program ÚNKP-22-5-ELTE-1154 New National Excellence Program of the Ministry for Culture and
    Innovation and the Hungarian National Research, Development and Innovation Office - NKFIH (grant no. K124749, no. K142993).
    
    G. Somlai is a Fulbright research fellow at the Graduate Center of the City University of New York. This research exchange program is also supported by the Magyar Állami Eötvös Ösztöndíj.
    
    G. Somlai is a J\'anos Bolyai Fellowship holder and supported by the OTKA grant no. SNN 132625.

    \end{document}